\documentclass{amsproc}
\usepackage{mathtools}
\usepackage{hyperref}

\newtheorem{theorem}{Theorem}[section]
\newtheorem{lemma}[theorem]{Lemma}
\newtheorem{pro}[theorem]{Proposition}
\newtheorem{cor}[theorem]{Corollary}

\theoremstyle{definition}
\newtheorem{definition}[theorem]{Definition}
\newtheorem{example}[theorem]{Example}

\theoremstyle{remark}
\newtheorem{remark}[theorem]{Remark}

\numberwithin{equation}{section}

\def\k{{\mathbb K}}

\newcommand{\der}{\operatorname{\mathsf{PDer}}}

\newcommand{\innder}{\operatorname{\mathsf{Ham}}}
\newcommand{\hh}{\operatorname{\mathsf{PH^1}}}
\newcommand{\ad}{\operatorname{\mathsf{Ham}}}
\newcommand{\wt}{\operatorname{\mathsf{wt}}}
\newcommand{\fract}{\operatorname{\mathsf{Fract}}}

\newcommand\rk{\operatorname{\mathsf{rk}}}

\newcommand{\ZC}{\operatorname{\mathsf{Z}}}
\newcommand{\NS}{\operatorname{\mathsf{N}}}

\def\ch{{\mathcal H}}

\newcommand{\Z}{\mathbb{Z}}
\newcommand{\Q}{\mathbb{Q}}

\begin{document}

\title{Poisson derivations and cohomology of Poisson nilpotent algebras}

\author{S Launois}
\address{Universit\'e de Caen Normandie, 
CNRS UMR 6139 LMNO, 14032 Caen, France}
\email{stephane.launois@unicaen.fr}

\author{S A Lopes}
\address{CMUP, Departamento de Matem\'atica, Faculdade de Ci\^encias, Universidade do Porto, Rua do Campo Alegre s/n, 4169--007 Porto, Portugal}
\email{slopes@fc.up.pt}

\author{I Oppong}
\address{School of Computing and Mathematical Sciences,
University of Greenwich, 
Old Royal Naval College, Park Row,
London SE10 9LS, UK}
\email{I.Oppong@greenwich.ac.uk}

\date{January 1, 1994 and, in revised form, June 22, 1994.}



\begin{abstract}
We compute the Poisson derivations and the first Poisson cohomology group of uniparameter Poisson Nilpotent Algebras (PNA for short) under the hypothesis that Poisson-normal elements are Poisson-central. This applies in particular to certain Poisson algebras associated to Bott--Samelson varieties.
\end{abstract}

\maketitle

\section{Introduction}

Our aim in this article is to initiate the study of Poisson (co)homology of Poisson Nilpotent Algebras  (PNA for short), also known as Poisson CGL extensions.  These Poisson polynomial algebras have remarkable properties; for instance, under some technical hypotheses, Goodearl and Yakimov have shown that they possess a Poisson cluster algebra structure~\cite{goodearl2023cluster}. Recent work of Lu and Matviichuk~\cite{LuMatviichuk} further emphasizes the relevance of Poisson cohomology in this setting: they show that symmetric Poisson CGL extensions can be understood as Poisson deformations of log-canonical Poisson structures, with the corresponding deformation and classification theory governed by suitable second Poisson cohomology groups (see also~\cite{S93} for the case of Poisson superalgebras). Examples of such PNAs include coordinate rings of Bott--Samelson varieties due to a result of Elek and Lu \cite{elek-lu}. 

In \cite{LLO25}, we computed the derivations and first Hochschild cohomology group of Quantum Nilpotent Algebras (QNA for short) under a technical assumption which covers the positive part of any quantized universal enveloping algebra of a complex simple Lie algebra, but which excludes some important classes of QNAs, such as quantum matrices.

Our aim in this article is to compute the Poisson derivations and first Poisson cohomology group for PNAs under the technical condition that Poisson-normal elements are Poisson-central. 
Examples of such PNAs include certain (coordinate rings of) Bott--Samelson varieties. We could relax our hypotheses so as to mimic the technical assumptions in \cite{LLO25}, but here we favoured a smoother exposition at the cost of the more restrictive assumption that Poisson-normal elements are Poisson-central. 

In order to compute Poisson derivations of a PNA $R=\k[x_1, \ldots, x_N]$, where $\k$ is a field of characteristic $0$, we use the initial Poisson cluster $(y_1,\dots,y_N)$ constructed by Goodearl and Yakimov. This has the property that 
$$\k[y_1,\dots,y_N]\subseteq R \subseteq \k[y_1^{\pm 1},\dots,y_N^{\pm 1}]$$
and $\k[y_1,\dots,y_N]$ is a Poisson affine space in the sense that its Poisson structure is log-canonical (i.e., there exist scalars $q_{ij}$ such that $\{y_i,y_j\}=q_{ij} y_{i}y_{j}$ for all $i,j$). The Poisson derivations of the associated Poisson torus $\k[y_1^{\pm 1},\dots,y_N^{\pm 1}] $ were computed in \cite{lo} (see also~\cite{LuMatviichuk}), where it was proved that any such Poisson derivation is the sum of a Hamiltonian derivation and a central derivation (i.e., a Poisson derivation that acts on each generator $y_i$ by multiplication by a Poisson-central element). Thus our strategy to describe Poisson derivations of $R$ is to extend them to the Poisson torus $\k[y_1^{\pm 1},\dots,y_N^{\pm 1}] $ so that we can use the results in \cite{lo} to describe their actions on the variables $y_i$. We then need to describe the action on the variables $x_i$ and this causes a number of issues. For instance, the Hamiltonian derivation appearing is associated to an element of the Poisson torus $\k[y_1^{\pm 1},\dots,y_N^{\pm 1}] $ and so \textit{a priori} it might not belong to the PNA $R$. Moreover, the Poisson-central elements appearing belong to the Poisson center of the Poisson torus and \textit{a priori} not to the Poisson center of $R$, which can be much smaller. So while we have some control on the action of a Poisson derivation on the $y_i$, this information is only partial as it involves parameters from a larger algebra. 

To overcome these difficulties, we introduce an intermediate localization $RE^{-1}$ of $R$ which has the same Poisson center as $R$ and which can also be described as a localization of $\k[y_1,\dots,y_N]$ at some of the $y_i$. This second description is key for us as it allows to describe its Poisson derivations from those of the associated Poisson torus. In particular, we show that each Poisson derivation of $RE^{-1}$ is a sum of a Hamiltonian derivation and a Poisson-central derivation (that is, a Poisson derivation that acts by multiplication by a Poisson central element on the non-Poisson-central variables $y_i$). The fact that the Poisson center of $RE^{-1}$ is equal to the Poisson center of $R$ is key as this later allows us to prove that the defining element of the Hamiltonian derivation actually belongs to $R$ (and not just $RE^{-1}$).

Our main results show that any Poisson derivation of $R$ is the sum of a Hamiltonian derivation of $R$ and a Poisson central derivation of $R$, acting on each $x_i$ by multiplication by a Poisson-central element of $R$. It turns out that this central multiple depends only on the homogeneous degree of $x_i$ (in the sense discussed in Subsection~\ref{SS:PPL:HPP}), whence the first Poisson cohomology group of $R$ is free over the Poisson center of $R$, with rank equal to the rank of $R$ as a PNA. Our main results apply in particular to various Poisson algebras related to Bott--Samelson varieties for which the Poisson Rigidity Theorem of Levitt-Yakimov \cite[Theorem B]{levitt-yakimov} shows that the Poisson automorphism group is (isomorphic to) the semi-direct product of a torus of rank the rank of the PNA by a finite group.

Due to results of Goodearl and Yakimov \cite{goodearl2023cluster}, many Poisson CGL have a Poisson cluster algebra structure (see also \cite{Lu2026}) and our article suggests computing Poisson derivations of Poisson cluster algebras. We will go back to these in future work. 

This article is organized as follows: after fixing notations about PNAs in Section 2, we recall in Section 3 the construction by Goodearl and Yakimov of an initial Poisson cluster. In Section 4, we focus on describing the Poisson derivations of the localization $RE^{-1}$ and in Section 5, we compute the Poisson centers of the various Poisson algebras involved. Finally, our main results are established in Section 6. \\

We work over an arbitrary base field $\k$ of characteristic $0$ and identify its prime subfield with $\Q$, the field of rational numbers. In particular, unless otherwise stated, all maps between $\k$-vector spaces and $\k$-algebras are assumed to be $\k$-linear and unadorned tensor products are over $\k$. 

As usual, an element of a list appearing with a hat is omitted from this list. Given integers $i, j\in\Z$, we set $[i,j]:=\{k\in\Z\mid i\leq k\leq j\}$ and let $\Z_{\geq 0}$ and $\Z_{> 0}$ be the sets of nonnegative and positive integers, respectively.

\section{Poisson polynomial algebras}
\subsection{Poisson algebras}

Fix a base field $\k$ of characteristic zero throughout. All algebras are assumed to be over $\k$, and all relevant maps (automorphisms, derivations, etc.) are assumed to be $\k$-linear.

$A$ Poisson algebra is a commutative, associative $\k$-algebra $R$ endowed with a skew-symmetric $\k$-bilinear map $\{-,-\}:R\times R\rightarrow R$ which satisfies both the Leibniz rule and the Jacobi identity. In particular, $\{-,-\}$ is a Lie bracket and is often referred to as the \textit{Poisson bracket} of $R$. 

There are many examples of Poisson algebras in the literature. For now, we just mention the specific examples of Poisson affine spaces and Poisson tori which will play a crucial role in this article. 

Given any skew-symmetric matrix $\Lambda=(\lambda_{ij})\in M_n(\k),$ one defines a (unique) Poisson bracket on the polynomial algebra  $R=\k[x_1, \ldots, x_n]$ via $\{x_i, x_j\}=\lambda_{ij}x_ix_j$. We call this Poisson algebra the \textit{Poisson affine space} associated to $\Lambda$. This Poisson structure on $R$ extends uniquely to a Poisson structure on $T:=\k[x_1^{\pm 1}, \ldots, x_n^{\pm 1}]$ (and in fact to any localization of $R$), which we call the \textit{Poisson torus} associated to $\Lambda$. 

We finish this paragraph with a bit of terminology: A \textit{Poisson ideal} of $R$ is an ideal $J$ such that $\{R, J\}\subseteq J$. A Poisson algebra $R$ is {\em Poisson-simple} if its only Poisson ideals are $0$ and $R$. The \textit{Poisson center} of a Poisson algebra $R$ is defined as
$$\ZC_P(R):=\{z\in R\mid \{z, R\}=0\}.$$
An element $c\in R$ os called \textit{Poisson-normal} if $cR$ is a Poisson ideal of $R$; equivalently, if $\{c, R\}\subseteq cR$. The set of Poisson-normal elements of $R$ is not in general closed under addition, so we define the \textit{Poisson-normal subalgebra} of $R$ as the subalgebra of $R$ generated by its Poisson-normal elements. This algebra is denoted $\NS_P(R)$ and consists of finite sums of Poisson-normal elements of $R$. Clearly, 
$\ZC_P(R)\subseteq \NS_P(R).$ 

A \textit{Poisson automorphism} of a Poisson algebra is an algebra automorphism that preserves the Poisson bracket. Similarly, a \textit{Poisson derivation} of a Poisson algebra is a ($\k$-linear) derivation for both the associative and the Lie structures. As usual, we denote by $\mathrm{PAut}(R)$  the group of Poisson automorphisms of $R$, and by $\mathrm{PDer}(R)$ its Lie algebra of Poisson derivations. A very natural and important class of Poisson derivations consists of the \textit{Hamiltonian derivations} (also known as Poisson inner derivations) $\innder_x:=\{x,-\}$, where $\innder_x (y):=\{x,y\}$ for all $x, y \in R$. Then $\innder(R):=\{\innder_x\mid x\in R\}$ is a Lie ideal of $\mathrm{PDer}(R)$ and the quotient Lie algebra $\hh(R):=\mathrm{PDer}(R)/\innder(R)$ is the first Poisson cohomology group of $R$.

\subsection{Poisson polynomial algebras}

Let $A$ be a Poisson $\k$-algebra and $\sigma, \delta:A\rightarrow A$ be  $\k$-linear maps. Then, by \cite[Theorem 1.1]{oh}, we can extend the Poisson structure on $A$ to a Poisson bracket on the polynomial algebra $R=A[x]$ in such way that $\{x,a\}=\sigma(a)x+\delta(a)$ for all $a\in A$ if and only if the following are satisfied:
\begin{itemize}
\item[1.] $\sigma (\{a,b\})=\{\sigma(a), b\}+\{a, \sigma(b)\}$ for all $a,b\in A$; that is, $\sigma \in \mathrm{PDer}(A)$.
\item[2.] $\delta(\{a, b\})=\{\delta(a), b\}+\{a, \delta(b)\}+\sigma(a)\delta(b)-\delta(a)\sigma(b)$ for all $a,b\in A$. In this case, we say that $\delta$ is a \textit{Poisson $\sigma$-derivation} of $A$.
\end{itemize}
If these conditions are satisfied, we say that $R$ is a  \textit{Poisson-Ore extension} of $A$, and we write $R=A[x; \sigma, \delta]_P$. This construction can easily be iterated to give rise to so-called iterated Poisson-Ore extensions over $\k$, which take the following form:
$$
R=\k[x_1]_P[x_2;\sigma_2, \delta_2]_P\cdots [x_N;\sigma_N, \delta_N]_P,
$$
where $\sigma_k$ and $\delta_k$ are Poisson derivations and  Poisson $\sigma_k$-derivations of $$R_{k-1}:=\k[x_1]_P[x_2;\sigma_2, \delta_2]_P\cdots [x_{k-1};\sigma_{k-1}, \delta_{k-1}]_P$$ for each $k\in [1,N].$ Note that $R_0=\k$ and $\sigma_1=\delta_1=0.$ 

We note that a Poisson affine space can be presented as an iterated Poisson-Ore extension with all Poisson $\sigma$-derivations equal to zero.

\subsection{Poisson nilpotent algebras}
In this article, we will mainly be concerned with Poisson nilpotent algebras (PNAs for short) (also known as Poisson CGL-extensions). These Poisson algebras were introduced by Goodearl and Yakimov in \cite{goodearl2023cluster} as Poisson analogues of the so-called quantum nilpotent algebras.

\begin{definition}\cite[Def. 5.1]{goodearl2023cluster}
\label{dcgl}
A \textit{Poisson nilpotent algebra} (PNA) is an iterated Poisson-Ore extension of the form 
\begin{equation}
\label{I.P-O.E-definition}
R=\k[x_1]_P[x_2;\sigma_2, \delta_2]_P\cdots [x_N;\sigma_N, \delta_N]_P,
\end{equation}
equipped with a rational Poisson action of a torus $\mathcal{H}=(\k^*)^m$ such that the following properties are satisfied:
\begin{itemize}
\item[(a)] the generators $x_1, \ldots, x_N$ are $\mathcal{H}$-eigenvectors;
\item[(b)] for every $k\in [2,N]$, $\delta_k$ is a locally nilpotent Poisson  $\sigma_k$-derivation of the subalgebra $R_{k-1}$ of $R$;
\item[(c)] For every $k\in [1,N]$, there exists $h_k \in \mathrm{Lie}(\mathcal{H}) $ (the Lie algebra of the algebraic group $\mathcal{H}$) such that  $\sigma_k$ is the restriction of $h_k$ to $R_{k-1}$ and the $h_k$-eigenvalue of $x_k$, to be denoted by $\lambda_k$, is nonzero.
\end{itemize}
\end{definition}
Note for later use that $x_j$ is an eigenvector for $\sigma_k$ ($j<k$). We denote by $\lambda_{kj}$ the corresponding eigenvalue, that is $\sigma_k (x_j)= \lambda_{kj} x_j$. We then say that the PNA is \textit{uniparameter} if $\lambda_k, \lambda_{kj}\in \Q$, for all $k, j\in [1, N]$, with $j<k$.

Observe that the group of invertible elements of a PNA is reduced to $\k^*$ since any PNA is a commutative polynomial algebra over $\k$. 

Examples of PNAs include Poisson affine spaces, Poisson matrix varieties and Poisson structures attached to Bott--Samelson varieties (see~\cite[Theorem 5.12]{elek-lu}). 

Let $n$ be the {\em rank} of the PNA $R$; that is, $\rk(R):=|\{i \in [1,N] \mid \delta_i=0\}|=n$. (The rank of $R$ is also equal to the number of height one Poisson-prime ideals of $R$ which are invariant under $\mathcal{H}$, see \cite[Sec. 5.2]{goodearl2023cluster} and Section~\ref{SS:PPL:HPP} below). It follows from \cite[Theorem 6.8]{goodearl2023cluster} that we can (and will) assume that $m=n$, so that $\mathcal{H}=(\k^*)^n$. In other words, we assume that $\mathcal{H}$ is the largest torus making $R$ into a PNA.

\begin{example}\label{E:B-S}
From \cite[5.3]{elek-lu}, for any sequence $\bf{u}$ of simple reflections in the Weyl group associated to a complex simple Lie algebra, the coordinate ring of the chart $\mathcal{O}^{\bf{u}}$ of the Bott--Samelson variety $Z_{\bf{u}}$ is a PNA. 
\end{example}


\section{Poisson-prime elements and localizations of PNAs}

In this section, we review the algorithmic construction, due to Goodearl--Yakimov \cite{goodearl2023cluster}, of homogeneous elements $y_1, \dots , y_N$ of a PNA $R$ of rank $n$ as in \eqref{I.P-O.E-definition}.  These elements are algebraically independent and the PNA $R$ is sandwiched between the polynomial algebra $\k [y_1, \dots y_N]$ and the associated Laurent polynomial ring $\k [y_1^{\pm 1}, \dots y_N^{\pm 1}]$.  Our strategy to compute the Poisson derivations of $R$ will be to first describe the action of derivations on the $y_i$. In this section, we also establish the necessary results to then later compute the action of a Poisson derivation on the $x_i$.

\subsection{Homogeneous Poisson-prime elements}\label{SS:PPL:HPP}

Let $X(\mathcal{H})$ denote the set of all rational characters of the torus $\mathcal{H}$. Then, $X(\mathcal{H})\simeq\Z^n$
is an abelian group called the \textit{character group} of $\mathcal{H}.$  The action of $\mathcal{H}$ on $R$ induces an $X(\mathcal{H})$-grading of $R.$ The $\mathcal{H}$-eigenvectors are exactly the nonzero homogeneous elements under this grading (see \cite[Sec. 1.4]{goodearl2023cluster}). 

The ideal $pR$ is \textit{Poisson-prime} if and only if it is a Poisson ideal and a prime ideal. A nonzero element $p\in R$ is said to be a \textit{Poisson-prime element} if  $pR$ is a Poisson-prime ideal (see \cite[Sec. 4.1]{goodearl2023cluster}).  Finally, a Poisson-prime element $p\in R$ that is also an  $\mathcal{H}$-eigenvector is called a \textit{homogeneous Poisson-prime element} or a \textit{Poisson-prime $\mathcal{H}$-eigenvector}.

The algorithmic construction due to the work  Goodearl and Yakimov \cite{goodearl2023cluster} of the homogeneous Poisson-prime elements relies on the existence of a \textit{colouring map} $\mu: [1, N] \rightarrow [1,n]$. Attached to such a map, one can define two functions, the \textit{predecessor} function $p=p_\mu:[1,N]\rightarrow [1,N]\sqcup \{-\infty\}$ and the \textit{successor} function $s=s_\mu:[1,N]\rightarrow [1,N]\sqcup \{+\infty\}$ by: 
$$
p(k)=\begin{cases}
\text{max} \ \{j<k\mid \mu(j)=\mu(k)\} & \text{if $\exists j<k$ such that $\mu(j)=\mu(k),$}\\
-\infty & \text{otherwise,}
\end{cases}
$$ and 
$$
s(k)=\begin{cases}
\text{min} \ \{j>k\mid \mu(j)=\mu(k)\} & \text{if $\exists j>k$ such that $\mu(j)=\mu(k),$}\\
+\infty & \text{otherwise.}
\end{cases}
$$

In \cite{goodearl2023cluster}, the authors construct a colouring map $\mu: [1, N] \rightarrow [1,n]$ and use it to describe the homogeneous Poisson-prime elements of a PNA. We recall their result below.

\begin{theorem} (\cite[Theorem 5.5]{goodearl2023cluster})
\label{homo}
Let $R$ be a PNA of rank $n$ as in \eqref{I.P-O.E-definition}. There exists a surjective function $\mu: [1, N] \rightarrow [1,n]$ such that the following homogeneous elements $y_1, \ldots, y_N$ of $R$ can recursively be constructed as follows:
\begin{equation}
\label{5T}
y_k:=\begin{cases}
y_{p(k)}x_k-c_k, & \text{if $p(k)\neq -\infty$},\\
x_k, & \text{if $p(k)= -\infty$},
\end{cases}
\end{equation}
for some homogeneous $c_k\in R_{k-1}.$ The elements $y_1, \ldots, y_N$  satisfy the  property that, for every $k\in [1,N],$ we have
\begin{equation}
\label{T5}
\{y_j\mid j\in [1,k], \  s(j)>k\}
\end{equation}
is the set of homogeneous Poisson-prime elements  of $R_k$, up to scalar multiples. In particular, $y_k$ is a homogeneous Poisson-prime element of $R_k$ and a prime element of $R$, for all $k\in [1, N].$
\end{theorem}

We record additional properties of the elements $y_k$ in the following remark.  

\begin{remark}\hfill
\label{T25}
\begin{itemize}
\item[1.] $\delta_k=0$ if and only if $p(k)=-\infty$ (see \cite[Theorem 5.5]{goodearl2023cluster}). 
\item[2.] Assume that $p(k) \neq -\infty$. Then:
\begin{enumerate}
\item it follows from \cite[Proposition 5.9]{goodearl2023cluster} that $c_k=\alpha_{k,p(k)}^{-1}\delta_k(y_{p(k)}),$ where $\sigma_k(y_{p(k)})=\alpha_{k,p(k)} y_{p(k)}$. 
\item $y_{p(k)}$ is a homogeneous Poisson-prime element of $R_{k-1}$ as $s(p(k))=k>k-1.$ Hence, $y_{p(k)}R_{k-1}$ is a  Poisson-prime ideal of $R_{k-1}.$ 
\item $c_k\not\in y_{p(k)}R_{k-1}$ (see \cite[Theorem 4.12(ii)]{goodearl2023cluster}).  
\end{enumerate}
\end{itemize}
\end{remark}

\subsection{Partially localized Poisson affine space associated to a PNA}\label{S:rec-QNA-localizations:SSplqas}

From \cite[Proposition 5.8]{goodearl2023cluster}, the subalgebra $\mathcal{A}$ of $R$ generated by the homogeneous elements $y_1, \ldots, y_N$ is a Poisson affine space associated to some  skew-symmetric matrix $Q:=(q_{ij})\in M_N(\k)$. Thus,

\begin{equation}
\label{b2}
\mathcal{A}:=\k[y_1, \ldots, y_N]
\end{equation}
is a Poisson affine space with $\{y_j, y_i\}=q_{ij}y_iy_j$, for all $i,j\in[1,N]$. The scalars $q_{ij}$ can be explicitly expressed in terms of sums of the $\lambda_{ij}$ (see \cite[5.18]{goodearl2023cluster} for more details). In particular, in case $R$ is uniparameter then all $q_{ij}\in \Q$. We denote by 
\begin{equation}
\label{torus}
\mathcal{T}:=\mathcal{A}[y_1^{-1}, \ldots, y_N^{-1}]=\k[y_1^{\pm 1}, \ldots, y_N^{\pm 1}]
\end{equation}
the Poisson torus associated to $\mathcal{A}$.

It was proved in  \cite[Proposition 5.8]{goodearl2023cluster} that the elements $y_1, \dots, y_N$ form an initial Poisson cluster for $R$ in the sense that we have the following tower of Poisson algebras: 
\begin{equation}
\label{T7}
\mathcal{A}\subseteq R\subseteq \mathcal{T}\subseteq \fract(R),
\end{equation} 
where $\fract(R)$ is the field of fractions of $R$. The relationship between $R$ on one hand and $\mathcal{A}$ (or $\mathcal{T}$) on the other hand is actually stronger, as we shall see below. 

As we will require localizations of $R$ at multiplicative sets generated by various subsets of $\{y_1, \ldots, y_N\}$, we introduce some notation. Given $I \subseteq [1,N]$, we set $Y_I:=\{y_{k}\mid k\in I\}$ and we denote by $E_I$ the multiplicative system of $R$ generated by $Y_I$. Moreover, let 
\begin{equation*}
\mathfrak{s}_{<+\infty}:=\{k\in [1, N]\mid s(k)<+\infty\} \quad\text{and}\quad \mathfrak{s}_{+\infty}:=\{k\in [1, N]\mid s(k)=+\infty\}
\end{equation*}
and set $Y_{<+\infty}:=Y_{\mathfrak{s}_{<+\infty}}$, $E_{<+\infty}:=E_{\mathfrak{s}_{<+\infty}}$, $Y_{+\infty}:=Y_{\mathfrak{s}_{+\infty}}$ and $E_{+\infty}:=E_{\mathfrak{s}_{+\infty}}$. We then deduce the last equality of the following statement from \cite[(10.1)]{goodearl2023cluster}: 
\begin{equation}
\label{R-torus}
R[E_{[1,N]}]^{-1}=R[y_1^{-1}, \ldots, y_N^{-1}]=\mathcal{T}.
\end{equation}
We will focus on a specific common localization of $\mathcal{A}$ and $R$: a straightforward induction using~\eqref{5T} shows that $RE_{<+\infty}^{-1}=\mathcal{A}E_{<+\infty}^{-1}$. More generally, suppose that $\mathfrak{s}_{<+\infty}\subseteq I\subseteq [1,N]$. Then $E_{<+\infty}\subseteq E_I$ and it is clear that we still have $RE_{I}^{-1}=\mathcal{A}E_{I}^{-1}$; so we have the following tower of algebras: 
 
\begin{equation}
\label{T20}
\mathcal{A} \subseteq R\subseteq RE_{I}^{-1}=\mathcal{A}E_{I}^{-1}\subseteq \mathcal{T}\subseteq \fract(R).
\end{equation}

This link between the PNA $R$ and the partially localized Poisson affine spaces $\mathcal{A}E_{I}^{-1}$, for appropriate choices of $\mathfrak{s}_{<+\infty}\subseteq I\subseteq [1,N]$, will allow us to attack the computation of Poisson derivations of PNAs by first computing the Poisson derivations of partially localized Poisson affine spaces and then bringing information from $RE_{I}^{-1}$ to $R$, thanks to the following result.

\begin{pro}
\label{b4}
Let $I,J \subseteq [1,N]$. Then $RE_I^{-1} \cap RE_J^{-1}=RE_{I\cap J}^{-1}$.
\end{pro}
\begin{proof}
This easily follows from the $y_i$'s being prime in the polynomial ring $R$, see \cite[Theorem 5.5]{goodearl2023cluster}. 
\end{proof}

\subsection{Poisson-normal elements cannot be multiples of non-Poisson-primes $y_i$}

We proceed with a result that proves that $0$ is the only Poisson-normal element that is a multiple of a non-Poisson-prime $y_i$ (that is, with $s(i)\neq +\infty$). This result will be used later, namely to describe the action of a Poisson derivation of $R$ on the generators $x_i$, after we control its action on the homogeneous elements $y_i$, see Proposition~\ref{centralderivation1}. 

\begin{lemma}\cite[Lemma 10.1]{goodearl2023cluster}
\label{dd}
For all $ i, j\in [1, N]$ with $i\neq j$, we have that $y_i\not\in y_jR.$ 
\end{lemma}

The set $Y_{+\infty}$ of homogeneous Poisson-prime elements of $R$
generates $\NS_P(R)$, the Poisson-normal subalgebra of $R$, which has Krull dimension $n$ (see \cite[Theorem 4.6]{goodearl2023cluster}). Thus, 
\begin{equation}
\label{T9}
\NS_P(R)=\k[ y_j \mid j\in \mathfrak{s}_{+ \infty} ]\subseteq \mathcal{A}
\end{equation}
is a Poisson affine space, associated to some skew-symmetric sub-matrix $Q'$  of $Q$. We remark that, although a sum of Poisson-normal elements may no longer be Poisson-normal, for all $u\in \NS_P(R)$ with $\mathcal{H}$-eigenvalue decomposition $u=u_1+\cdots+u_m$, each $u_i$ is Poisson-normal.

We are now ready to establish the following technical remarks.

\begin{pro}
\label{b5}
Let $R$ be a PNA and  $y_i\in R$ be a homogeneous element with $s(i)<+\infty$. Then $\NS_P(R) \cap y_iR = \{0\}$.
\end{pro}
\begin{proof}
Suppose that $0\neq u\in \NS_P(R) \cap y_iR$. Standard arguments show that, without loss of generality, we can assume that $u=y_iv$ with $u$ and $v$ both $\ch$-eigenvectors and $u$ Poisson-normal. From \cite[Theorem 5.2]{goodearl2023cluster}, we have that the PNA $R$ is an $\ch$-Poisson-UFD, and so it follows from 
\cite[Proposition 4.1]{goodearl2023cluster} that $u= \lambda \prod_{s(j)=+\infty} y_j^{e_j}$, with $\lambda\in\k^*$ and $e_j\in\Z_{\geq 0}$, since the homogeneous Poisson-prime elements of $R$ are the $y_j$ with $s(j)=+\infty$. Hence, $y_iv=  \lambda \prod_{s(j)=+\infty} y_j^{e_j}$ and the primeness of $y_i$ in $R$ (see~\cite[Theorem 5.5]{goodearl2023cluster})  implies that there exists $j$ such that $s(j)=+\infty$ and $i=j$. This contradicts our hypothesis that $s(i)<+\infty$. 
\end{proof}

\begin{pro}
\label{proposition:center in normal}
$\ZC_P(\mathcal{T})= \ZC_P(RE_{+\infty}^{-1})=\{z\in \NS_P(R)E_{+\infty}^{-1}\mid \{R, z\}=0 \}.$
\end{pro}
\begin{proof}
Observe that $\mathcal{T}=RE_{<\infty}^{-1}E_{+\infty}^{-1}=RE_{+\infty}^{-1}E_{<\infty}^{-1}$. The chain of inclusions $\{z\in \NS_P(R)E_{+\infty}^{-1}\mid \{R, z\}=0 \}\subseteq \ZC_P(RE_{+\infty}^{-1}) \subseteq \ZC_P(\mathcal{T})$ is trivial. For the reverse inclusion, it is enough to prove that 
$\ZC_P(\mathcal{T}) \subseteq \{z\in \NS_P(R)E_{+\infty}^{-1}\mid \{R, z\}=0\}.$

 Note that $\ZC_P(\mathcal{T})$ is the $\k$-span of the Poisson-central  monomials in the generators $y_1^{\pm 1}, \ldots, y_N^{\pm 1}$ of $\mathcal{T}.$ Hence, without loss of generality, we can assume that $w\in \ZC_P(\mathcal{T})$ is a Poisson-central Laurent monomial. Therefore, $w$ is an $\mathcal{H}$-eigenvector. Now, the set 
$J:=\{r\in R \mid wr\in R\}$ is a nonzero Poisson $\mathcal{H}$-invariant ideal of $R$. By a standard noetherianity argument, we obtain the existence of  $c\in J \cap E_{+\infty}$. Since $\{w, R\}=0$ and $c\in E_{+\infty} \subseteq \NS_P(R),$ we have that $a:=wc \in \NS_P(R).$ Hence  $w=ac^{-1} \in \{z\in \NS_P(R)E_{+\infty}^{-1}\mid \{R, z\}=0\}$, as desired. 
\end{proof}


\section{Structure theorem for Poisson derivations of partially localized Poisson affine spaces}
\label{section: derivations partially localized P-affine spaces}
An important invariant of a Poisson algebra $A$ is its Poisson cohomology. In low degrees, these Poisson cohomology groups have concrete interpretations. For instance, the Poisson center of $A$ is its Poisson cohomology group of degree zero, and  the Poisson cohomology group of degree one can be identified with the Lie algebra of Poisson derivations modulo its ideal $\innder(A)=\{\innder_x\mid x\in A\}$ of Hamiltonian derivations.  

Our aim in this section is to compute the Poisson derivations of localizations of a Poisson affine space at the multiplicative system generated by a specific subset of the generators. Before doing so, we recall that Poisson derivations of Poisson tori $\mathcal{T}:=\k[x_1^{\pm 1}, \ldots, x_n^{\pm 1}]$ were computed in \cite[Theorem 2.6]{lo}, where it was proved that every Poisson derivation $D$ of $\mathcal{T}$ can be expressed uniquely as
$$D=\innder_x+ \delta,$$
for some $x\in \mathcal{T}$,
where $\delta$ is a central Poisson derivation of $\mathcal{T}$, that is, a Poisson derivation that acts on the canonical generators $x_i$ of the Poisson torus as multiplication by Poisson-central elements. 
The results of this section are the first step towards our goal to compute the Poisson derivations of a PNA via the tower of Poisson algebras given in~\eqref{T20}.

\begin{theorem}
\label{T:generic}
Let $A$ be a finitely generated Poisson $\k$-algebra such that
\begin{align*}
\der(A)=\innder(A)\oplus M,
\end{align*}
where $M$ is a $\ZC_P(A)$-module. Let $R=A[x_1, \ldots, x_m]$ be a Poisson affine space over $A$, where $x_1, \ldots, x_m$ are Poisson-central elements of $R$. Then:
\begin{enumerate}
\item $\ZC_P(R)=\ZC_P(A)[x_1, \ldots, x_m]\simeq \ZC_P(A)\otimes_{\k}\k[x_1, \ldots, x_m]$;
\item $\der (R) = \innder (R) \oplus \overline M \oplus \bigoplus_{j=1}^m \ZC_P(R) \partial_j$, where
\begin{itemize}
\item $\overline M=\ZC_P(R)M\simeq \ZC_P(R)\otimes_{\ZC_P(A)} M\simeq \k[x_1, \ldots, x_m]\otimes_\k M$, so that $D\in M\subseteq\der(A)$ is extended to a Poisson derivation of $R$ by setting $D(x_i)=0$, for all $i\in [1,m]$;
\item $\partial_j$ is the derivation of $R$ defined by $\partial_j(A)=0$ and $\partial_j(x_i)=\delta_{ij}$, for all $i, j\in [1,m]$.
\end{itemize}
\end{enumerate}
\end{theorem}
\begin{proof}
For all $\underline{\alpha}=(\alpha_1, \dots, \alpha_m) \in (\Z_{\geq 0})^m$, we set $ \underline{x}^{\underline{\alpha}}:=x_1^{\alpha_1} \cdots x_m^{\alpha_m}$.

First, let $z=\sum z_{\underline{\alpha}} \underline{x}^{\underline{\alpha}}$ be a Poisson-central element of $R$, where the sum runs over all $\underline{\alpha}=(\alpha_1, \dots, \alpha_m) \in (\Z_{\geq 0})^m$ and where all but a finite number of  $z_{\underline{\alpha}} \in A$ are zero. 
Since $\{x_i, x_j\}=0$ for all $i,j\in [1,m]$, one can easily check that $z$ is central if and only if $\{z , a\} = 0$ for all $a\in A$, that is, if and only if $\{z_{\underline{\alpha}}, a\}=0$ for all $a\in A$ and all $z_{\underline{\alpha}}$. 
Thus, $z$ is Poisson-central if and only if $z_{\underline{\alpha}} \in \ZC_P(A)$ for all $\underline{\alpha}$, and the first claim follows. 

It is easy to check that the $\partial_j$ define Poisson derivations of $R$ and that a  Poisson derivation $D$ of $A$ can be uniquely extended to a Poisson derivation of $R$ by setting $D(x_i)=0$ for all $i\in [1,m]$. Moreover, under such an extension, $zD\in\der(R)$ for all $z\in \ZC_P(R)$.

Now, let $D\in\der(R)$. Since $D(\ZC_P(R))\subseteq \ZC_P(R)$, we have $z_j:= D(x_j)\in \ZC_P(R)$ for all $j\in [1,m]$. Thus, replacing $D$ with $D-\sum_{j=1}^m z_j \partial_j\in\der(R)$, we can assume, without loss of generality, that $D(x_i)=0$ for all $i\in[1,m]$. For $a\in A$, we can write \[D(a)=\sum_{\underline{\alpha}} D_{\underline{\alpha}}(a)\underline{x}^{\underline{\alpha}},\] a finite sum with $D_{\underline{\alpha}}(a)\in A$ for all ${\underline{\alpha}}$. It is straightforward to check that the maps $D_{\underline{\alpha}}$ are in fact Poisson $\k$-derivations of $A$. Since $A$ is finitely generated as a $\k$-algebra, it follows that there is a finite set $K\subseteq (\Z_{\geq 0})^m$ such that
\begin{align*}
D=\sum_{\underline{\alpha}\in K} D_{\underline{\alpha}} \underline{x}^{\underline{\alpha}},
\end{align*}
where $D_{\underline{\alpha}}$ is extended to a Poisson derivation of $R$ as explained above, with $D_{\underline{\alpha}}(x_i)=0$ for all $i\in [1,m]$. By hypothesis, for each $\underline{\alpha}\in K$, there exist $u_{\underline{\alpha}}\in A$ and $E_{\underline{\alpha}}\in M$ such that $D_{\underline{\alpha}}=\ad_{u_{\underline{\alpha}}}+E_{\underline{\alpha}}$. Putting all of these together shows that
\begin{align*}
D=\ad_u+\sum_{\underline{\alpha}\in K} E_{\underline{\alpha}}\underline{x}^{\underline{\alpha}} \in\innder(R)+\overline M,
\end{align*}
where $u=\sum_{\underline{\alpha}\in K}u_{\underline{\alpha}}\underline{x}^{\underline{\alpha}} \in R$.

At this stage, we have proved that
\begin{align*}
\der (R) = \innder (R) + \overline M + \sum_{j=1}^m \ZC_P(R) \partial_j.
\end{align*}

To prove the direct sum decomposition in the statement, assume that 
\begin{align}\label{E:1:T:general}
\ad_u+E+ \sum_{j=1}^m z_j \partial_j=0,
\end{align}
with $u\in R$, $z_j\in \ZC_P(R)$ and $E=\sum_{\underline{\alpha}} E_{\underline{\alpha}}\underline{x}^{\underline{\alpha}} $, a finite sum with $E_{\underline{\alpha}}\in M$ for every ${\underline{\alpha}}$. Evaluating~\eqref{E:1:T:general} at $x_k$ leads to $z_k=0$, for all $k\in [1,m]$.

Write $u=\sum_{\underline{\alpha}}u_{\underline{\alpha}}\underline{x}^{\underline{\alpha}}$, a finite sum with $u_{\underline{\alpha}}\in A$, for all ${\underline{\alpha}}$. It follows that 
$$\sum_{\underline{\alpha}}\left(\ad_{u_{\underline{\alpha}}}+E_{\underline{\alpha}} \right) \underline{x}^{\underline{\alpha}}=0.$$
 Since $u_{\underline{\alpha}}\in A$ and $E_{\underline{\alpha}}\in M$, we have $\left(\ad_{u_{\underline{\alpha}}}+E_{\underline{\alpha}} \right)(A)\subseteq A$, thus evaluating at an arbitrary $a\in A$ we deduce that $\ad_{u_{\underline{\alpha}}}+E_{\underline{\alpha}}=0$ as a Poisson derivation of $A$, for all $\underline{\alpha}$. Now, from $\der(A)=\innder(A)\oplus M $, we deduce that $\ad_{u_{\underline{\alpha}}}=0=E_{\underline{\alpha}}$ as Poisson derivations of $A$. But also $\ad_{u_{\underline{\alpha}}}(x_i)=0=E_{\underline{\alpha}}(x_i)$ for all $i\in [1,m]$, so $\ad_{u_{\underline{\alpha}}}=0=E_{\underline{\alpha}}$ as Poisson derivations of $R$, for all $\underline{\alpha}$. We conclude that
\begin{align*}
\ad_u= \sum_{\underline{\alpha}}\underline{x}^{\underline{\alpha}}\ad_{u_{\underline{\alpha}}}=0=
\sum_{\underline{\alpha}}\underline{x}^{\underline{\alpha}} E_{\underline{\alpha}}=E,
\end{align*}
as desired. 
\end{proof}

Recall that a Poisson torus $\mathcal{T}$ is simple if and only if its center is reduced to $\k$, by \cite[Corollary 2.7]{lo}.

\begin{cor}
\label{prop: partially-localized-q-affine-spaces}
Let $\mathcal{T}:=\k[t_1^{\pm 1}, \ldots, t_n^{\pm 1}]$ be a simple Poisson torus. Set $R:= \mathcal{T}[x_1, \dots , x_m]$, a Poisson affine space over $\mathcal{T}$, where $x_1, \ldots, x_m$ are Poisson-central elements of $R$. Then: 
\begin{enumerate} 
\item $\ZC_P(R)= \k[x_1, \dots, x_m]$;
\item $\der (R) = \innder (R) \oplus \bigoplus_{i=1}^n \ZC_P(R) D_i \oplus \bigoplus_{j=1}^m \ZC_P(R) \partial_j$, where $D_i$ and $\partial_j$ are the Poisson derivations of $R$ defined by: 
$$D_i(t_k)= \delta_{ik}t_k \mbox{ and } D_i(x_k)=0;$$
 $$\partial_j(t_k)= 0 \mbox{ and } \partial_j(x_k)=\delta_{jk}.$$
\end{enumerate}
\end{cor} 
\begin{proof}
The proof follows directly from Theorem~\ref{T:generic} applied to $A=\mathcal{T}$, by noting that $\mathcal{T}$, being simple, has trivial center, and invoking \cite[Corollary 2.7]{lo}, which shows that every Poisson derivation of a simple Poisson torus is uniquely the sum of an inner Poisson derivation and a scalar Poisson derivation, that is, a Poisson derivation that acts by scalar multiplication on the generators of the Poisson torus.
\end{proof}

\section{Poisson centers}
\label{section: centers}

The Poisson center of a Poisson algebra is its Poisson cohomology group of degree zero and it constitutes an important invariant subalgebra which acts on its Poisson Lie algebra of Poisson derivations and on the first Poisson cohomology group.
In this section, we will be concerned with the Poisson centers of the PNA $R$, the Poisson affine space $\mathcal{A}$, the Poisson torus $\mathcal{T}$, and of certain localizations of these.

The first observation is that, since
\begin{equation*}
\ZC_P(R)\subseteq  \NS_P(R)\subseteq \mathcal{A}\subseteq R\subseteq \mathcal{T},
\end{equation*}
where $\NS_P(R)$ is the Poisson-normal subalgebra of $R$ introduced in~\eqref{T9}, it follows that 
\begin{equation}\label{E:chain:center:R:Aq:Tq}
\ZC_P(R)=\ZC_P(\mathcal{A})=\ZC_P(\mathcal{T})\cap \mathcal{A}.
\end{equation}

Note  that $\ZC_P(\mathcal{T})=\k[z_1^{\pm 1}, \ldots, z_\ell^{\pm 1}]$, for some $0\leq\ell\leq n$, and the $z_i$ can be chosen to be monomials in the $y_j^{\pm 1}$, with $j\in \mathfrak{s}_{+\infty}$ (see Proposition~\ref{proposition:center in normal}). In case $\ZC_P(\mathcal{T})=\k$, which is a possibility, the convention is that $\ell=0$.

We are looking for situations in which we are able to conclude, among other properties, that $\ZC_P(\mathcal{A})=\k[z_1, \ldots, z_\ell]$.

\begin{example}
Let $R=\mathcal A$ be the Poisson affine space associated to the matrix $\Lambda=
\begin{psmallmatrix}
0& 2 & 3\\
-2 &0 & 5\\
-3 &-5 &0 \\ 
\end{psmallmatrix}.
$  Then $R$ is a PNA of rank $3$ with $\ZC_P(\mathcal{T})=\k[z^{\pm 1}]$, with $z=x_1^5 x_2^{-3} x_3^{2}$, so $\ZC_P(\mathcal{A})=\k$.
\end{example}

In contrast with the example above, with many other PNAs, including some of the PNAs related to Bott--Samelson  varieties (see~\cite{elek-lu} or example \ref{E:B-S}), it is possible to choose the generators $z_i$ of the  Poisson-Laurent polynomial ring $\ZC_P(\mathcal{T})$ so that $\ZC_P(\mathcal{A})=\k[z_1, \ldots, z_\ell]$ (see Remark~\ref{R:rem:on:hyp} below).

To avoid technicalities, which may obscure the general idea, \textbf{we will assume that all Poisson-normal elements of $R$ are Poisson-central, i.e.\ that $\NS_P(R)=\ZC_P(R)$}. Then it follows from Proposition~\ref{proposition:center in normal} that $$\ZC_P(\mathcal{T})=\k[y_j^{\pm 1}\mid j\in \mathfrak{s}_{+\infty}]$$ is the Poisson torus associated to the normal subalgebra $\NS_P(R)$. Consequently, $\ZC_P(R)=\ZC_P(\mathcal{A})=\NS_P(R)$, $\ell=n=\rk(R)$ and we can take $\{z_1, \ldots, z_n\}=Y_{+\infty}$.

\begin{remark}[\textit{cf.} Example~\ref{E:B-S}]\label{R:rem:on:hyp}\hfill
\begin{enumerate}
\item The hypothesis $\NS_P(R)=\ZC_P(R)$ covers the coordinate rings of $\mathcal{O}^{\mathbf{u}}$ when $\bf{u} =w_0=-1$ (where as usual $w_0$ denotes the longest element of the Weyl group $W$ of a complex simple Lie algebra). See Example \ref{E:B-S} for more details.
\item A more general situation can be considered so as to cover all the cases of the PNAs $\mathcal{O}^{\mathbf{u}}$ when $\bf{u} =w_0$ for $\mathfrak{g}$ of arbitrary finite type, see~\cite[Section 4]{LLO25} for details. The methods employed there can be easily transposed and adapted to the PNA setting, at the cost of several technicalities.
\end{enumerate}
\end{remark}

So assume that $\NS_P(R)=\ZC_P(R)$. For ease of notation, set $E:=E_{<+\infty}$, the multiplicative set in $R$ generated by all the $y_i$ that are not Poisson-normal in $R$. Set
\begin{equation*}
\widehat{\mathcal{T}}:=\k[y_i^{\pm 1}\mid i\in \mathfrak{s}_{<+\infty}],
\end{equation*}
the Poisson torus of rank $N-n$ generated by the variables in $Y_{<+\infty}$, associated to the appropriate skew-symmetric submatrix $\widehat Q$ of $Q=(q_{ij})\in M_N(\k)$. Finally, set $\widehat{R}=RE^{-1}$.

\begin{pro}\label{P:ZCs}
Assume that $R$ is a uniparameter PNA with $\NS_P(R)=\ZC_P(R)$. Then we have the following:
\begin{enumerate}
\item $\widehat{R}=\mathcal{A}E^{-1}\simeq\widehat{\mathcal{T}}\otimes\ZC_P(R)$;\label{P:ZCs:1}
\item $\mathcal{T}\simeq\widehat{\mathcal{T}} \otimes\ZC_P(\mathcal T)$;\label{P:ZCs:2}
\item $\ZC_P(\widehat{\mathcal{T}})=\k$;\label{P:ZCs:3}
\item $\ZC_P(B)=\ZC_P(R)$, for any Poisson subalgebra $B$ such that $\mathcal{A} \subseteq  B\subseteq \widehat R$.\label{P:ZCs:4} 
\end{enumerate}
In particular, $N-n$ is even.
\end{pro}
\begin{proof}
For~(\ref{P:ZCs:1}) above, recall that we have observed at the end of Subsection~\ref{S:rec-QNA-localizations:SSplqas} that $\widehat{R}=RE^{-1}=\mathcal{A}E^{-1}=\mathcal{\widehat{T}}[y_i\mid i\in \mathfrak{s}_{+\infty}]\simeq \mathcal{\widehat{T}}\otimes \NS_P(R)=\mathcal{\widehat{T}}\otimes\ZC_P(R)$. Now~(\ref{P:ZCs:2}) follows from~(\ref{P:ZCs:1}), as 
\begin{equation*}
\mathcal{T}= \mathcal{A}E^{-1}E_{+\infty}^{-1}\simeq \mathcal{\widehat{T}}\otimes \ZC_P(R)E_{+\infty}^{-1}=
\widehat{\mathcal{T}} \otimes\ZC_P(\mathcal T).
\end{equation*}
To show~(\ref{P:ZCs:3}) note that, by~(\ref{P:ZCs:2}), 
\begin{equation*}
\ZC_P(\mathcal{T})\simeq\ZC_P(\widehat{\mathcal{T}})\otimes \ZC_P(\mathcal T).
\end{equation*}
So indeed it must be that $\ZC_P(\widehat{\mathcal{T}})=\k$. 

By the uniparameter hypothesis, $\widehat{\mathcal{T}}$ is a Poisson torus associated with a skew-symmetric matrix $\widehat Q\in M_{N-n}(\Q)$. If $N-n$ were odd, then $\det \widehat Q=0$, which would imply the existence of a nonzero vector $\underline\alpha\in\Z^{N-n}$ in the kernel of $\widehat Q$. This would lead to the existence of a nontrivial monomial in the $Y_{<+\infty}$ being central, which would contradict $\ZC_P(\widehat{\mathcal{T}})=\k$. Thence,  $N-n$ must be even.

It remains to prove~(\ref{P:ZCs:4}). By~(\ref{P:ZCs:1}) and~(\ref{P:ZCs:3}), $\ZC_P(\widehat R)=\ZC_P(R)=\ZC_P(\mathcal{A})$. 
If $\mathcal{A}\subseteq B\subseteq \widehat R$ is a chain Poisson subalgebras, then we deduce from $\widehat{R}=\mathcal{A} E^{-1}$ that $\ZC_P(\mathcal{A})\subseteq \ZC_P(B)\subseteq \ZC_P(\widehat R)$, yielding $\ZC_P(B)=\ZC_P(R)$.
\end{proof}


\section{The first Poisson cohomology group of a PNA}
\label{section: derivation QNA}

This is the main section of the paper and it focuses on investigating the first Poisson cohomology group of a uniparameter PNA $R$ satisfying $\NS_P(R)=\ZC_P(R)$.
Interesting examples of such PNAs are the semiclassical limits $\widehat{U^+(\mathfrak{g})}$ of $U_q^+(\mathfrak{g})$ with $\mathfrak{g}$ a finite-dimensional complex simple Lie algebra of type 
$A_1$, $B_n~(n\geq 2)$, $C_n~(n\geq 3)$, $D_n~(n\geq 4 ~\text{even})$, $ G_2$, $F_4$, $E_7$ and $E_8$ (see Example~\ref{E:B-S}). The main results still hold for the remaining finite types, although that requires a more technical set-up like that in~\cite[Section 4]{LLO25} (see Remark~\ref{R:rem:on:hyp}). 

We will show that each Poisson derivation of $R$ decomposes (uniquely) as a sum of an inner Poisson (Hamiltonian) derivation and a homogeneous Poisson derivation (see Subsection~\ref{SS:homogeneous}).

\subsection{The inner component of a Poisson derivation of $R$}

To study the space $\der(R)$ of Poisson $\k$-derivations of $R$, notice that we can uniquely extend any Poisson derivation $D$ of $R$ to a Poisson derivation of $\widehat{R}=RE^{-1}$, via localization. Thus, we identify $\der(R)$ with $\{D\in \der(\widehat R)\mid D(R)\subseteq R\}$.

From Proposition~\ref{P:ZCs} we have that $\widehat{\mathcal{T}}:=\k[y_i^{\pm 1}\mid i\in \mathfrak{s}_{<+\infty}]$ is a simple Poisson torus and $\widehat{R}=\widehat{\mathcal{T}}[y_j\mid j\in \mathfrak{s}_{+\infty}]$, a polynomial extension of $\widehat{\mathcal{T}}$ by $n$ Poisson-central variables. Thus, Corollary~\ref{prop: partially-localized-q-affine-spaces} shows that, as a Poisson derivation of $\widehat{R}$, we can decompose $D$ (uniquely) as
\begin{equation}\label{E:D-as-ad-plus-central}
D=\ad_x+\theta, 
\end{equation}
for some $x\in \widehat{R}$ so that, for all $i\in \mathfrak{s}_{<+\infty}$, $\theta(y_i)= \omega_i y_i$, for some $\omega_i \in \ZC_P(\widehat{R})=\ZC_P(R)$. If $i\in \mathfrak{s}_{+\infty}$ then, by our hypothesis that $\NS_P(R)=\ZC_P(R)$, it follows that $y_i\in\ZC_P(R)=\ZC_P(\widehat R)$. Evaluating~\eqref{E:D-as-ad-plus-central} at $y_i$ thus gives $\theta(y_i)=D(y_i)\in\ZC_P(R)$, as $D\in\der (R)$ and Poisson derivations take Poisson-central elements to Poisson-central elements.

Our immediate goal is to show that $x\in R$ and, for that purpose, we need to introduce intermediate subalgebras between $R$ and $\widehat R$.

Recall that $E=E_{<+\infty}$. For each $k\in \mathfrak{s}_{<+\infty}$, let 
$F_k:=E_{\mathfrak{s}_{<+\infty}\setminus\{k\}}$ be the multiplicative set in $R$ generated by $Y_{<+\infty}\setminus\{y_k\}$ and 
\begin{equation*}
B_k:=RF_k^{-1} 
\end{equation*}
be the corresponding localization.

Note that $y_k$ generates a multiplicative system in $B_k$. We have the following chain of embeddings: 
\begin{equation}
\label{T23}
 R\subseteq B_k\subseteq \widehat{R}=B_k[y_k^{-1}]\subseteq \mathcal{T}.
\end{equation}
We know already that $\ZC_P(R)=\ZC_P(B_k)=\ZC_P(\widehat{R})$, by Proposition~\ref{P:ZCs}, so we have complete control over the centers of all Poisson algebras appearing in~\eqref{T23}.

For each $k\in \mathfrak{s}_{<+\infty}$, let $\mathcal{Q}_{k}=\k[y_i^{\pm 1}\mid  i\in \mathfrak{s}_{<+\infty}\setminus\{k\}]$. So $\mathcal{Q}_{k}$ is a Poisson torus associated to a  skew-symmetric matrix ${\widehat Q}_k$  obtained from $Q$ by deleting its rows and columns indexed by $\mathfrak{s}_{+\infty} \cup\{k\}$. Moreover, $\mathcal{Q}_{k}\subseteq B_k$ and 
\begin{equation}\label{E:Tqh:over:Qk}
\widehat{\mathcal{T}}=\bigoplus_{j\in\mathbb{Z}}\mathcal{Q}_{k} y_k^j. 
\end{equation}

The rank of $\mathcal{Q}_{k}$ is $N-n-1$, which is odd, by Proposition~\ref{P:ZCs}. Thus, as argued before, since the center of a uniparameter odd rank Poisson torus is nontrivial and Poisson-central elements in a Poisson torus are sums of Poisson-central (Laurent) monomials in the generators of the Poisson torus, we have the following result.
 
\begin{lemma}
\label{a17}
For each $k\in \mathfrak{s}_{<+\infty}$, there exists a nontrivial monomial $\prod_{i\in \mathfrak{s}_{<+\infty}\setminus\{k\}} y_i^{m_i}$ (where at least one of the integers $m_i$ is nonzero) in the Poisson center of the  Poisson torus 
$\mathcal{Q}_{k}$.
\end{lemma}

Since $B_k$ is a localization of $R$ contained in $\widehat{R}$, we can further think of $D$ as a Poisson derivation of $\widehat{R}$ such that $D(R)\subseteq R$ and $D(B_k)\subseteq B_k$.

\begin{lemma}\label{L:inner:x-in-R}
Let $x\in \widehat{R}$ be as in~\eqref{E:D-as-ad-plus-central}. Then $x\in R$.
\end{lemma}

\begin{proof}
For each $k\in \mathfrak{s}_{<+\infty}$,
let $C_k=\mathcal{Q}_{k}[y_j\mid j\in \mathfrak{s}_{+\infty}]$. Then $C_k$ is a Poisson subalgebra of $B_k$ and, by \eqref{E:Tqh:over:Qk},
\begin{equation}\label{E:Rhat:over:Ck}
\widehat{R}=\bigoplus_{j\in\mathbb{Z}}C_k y_k^j. 
\end{equation}
As a result, $x\in \widehat{R}$ can be written uniquely as
$$x=\sum_{j\in \Z}a_{(k,j)}y_{k}^j,$$ 
where $a_{(k,j)}\in C_{k}.$
Decompose $x=x_+ + x_-,$ where
$$x_-=\sum_{j<0}a_{(k,j)}y_{k}^j \quad \text{and} \quad 
x_+=\sum_{j\geq 0}a_{(k,j)}y_{k}^j.$$
Clearly, $x_+\in B_k$.   We now proceed to show that $x_-\in B_k$ by using a strategy already used in \cite[Lemma 5.7]{lo} and \cite[Lemma 5.2]{LLO25}. Since $C_k$ is generated by the Poisson torus $\mathcal{Q}_{k}$  and the central variables $y_i$, with $i\in \mathfrak{s}_{+\infty}$, we deduce from Lemma~\ref{a17} that there exists a nontrivial monomial in the generators of $\mathcal{Q}_k$, denoted by $u_k$, that is Poisson-central in $C_{k}$. Note that $u_k$ does not belong to $\ZC_P(\widehat{R})$ because $\mathcal{Q}_{k}\cap\ZC_P(\widehat{R})\subseteq \widehat{\mathcal{T}}\cap \ZC_P(R)=\k$. Since the monomial $u_k$ is Poisson-central in $C_k$ and not in $\widehat{R}$, then~\eqref{E:Rhat:over:Ck} forces $\{y_k, u_k\}\neq 0$ and so $\{y_{k}, u_{k}\}=\xi u_{k}y_{k}$, for some $\xi:=\xi_{k}\in\k\setminus\{0\}$.

Since  $\theta(y_{j})=\omega_{j}y_{j}$
for each $j\in \mathfrak{s}_{<+\infty}$, with $\omega_{j}\in \ZC_P(R)=\ZC_P(B_k)$,  we have that $\theta(u_{k})=\eta_{k} u_{k}$, for some $\eta_{k}\in \ZC_P(B_k)$. 
Note that $u_{k}^{\pm 1}\in \mathcal{Q}_k\subseteq  B_{k}$.  Since $D$ restricts to a Poisson derivation of $B_k$, we have that
\[D(u_{k})=\ad_{x}(u_k)+\theta(u_{k})
=\ad_{x_-}(u_{k})+\ad_{x_+}(u_{k})+
\eta_{k} u_{k}\in B_{k}.\] 
Observe that $\ad_{x_+}(u_{k})+
\eta_{k} u_{k}\in B_{k}.$ Hence,  $\ad_{x_-}(u_{k})\in B_{k}.$
It follows that
$$\ad_{x_-}(u_{k})=\sum_{j=-1}^{-m} a_{(k,j)} j\xi y_{k}^ju_{k}\in B_{k},$$ 
for some $m\in \Z_{>0}.$ Let $n\in \Z_{>0}$ and set 
$$\ell^{(n)}=\{\{\ldots \{x_-, \underbrace{u_k\}, u_k\}, \ldots, u_k\}}_{\text{$n$ times}}\in B_{k}.$$ One can easily prove by induction on $n$ that 
$$\ell^{(n)}=\sum_{j=-1}^{-m} a_{(k,j)} (j\xi)^n y_{k}^ju_{k}^n.$$
Set $\chi_n:=\ell^{(n)}u_{k}^{-n}.$ Then,
$$\chi_n =\sum_{j=-1}^{-m} a_{(k,j)} (j\xi)^n y_{k}^j \in B_k,$$ for each $n \in [1, m].$

We have the following matrix equation:

\[
\begin{bmatrix}
-\xi&-2\xi&\cdots &-m\xi \\
(-\xi)^2&(-2\xi)^{2}&\cdots &(-m\xi)^{2}\\
(-\xi)^3&(-2\xi)^{3}&\cdots &(-m\xi)^{3}\\
\vdots&\vdots&\ddots&\vdots\\
(-\xi)^m&(-2\xi)^{m}&\cdots &(-m\xi)^{m}\\
\end{bmatrix}
\begin{bmatrix}
a_{(k,-1)}y_{k}^{-1}\\
a_{(k,-2)}y_{k}^{-2}\\
a_{(k,-3)}y_{k}^{-3}\\
\vdots\\
a_{(k,-m)}y_{k}^{-m}\\
\end{bmatrix}=
\begin{bmatrix}
\chi_1\\
\chi_2\\
\chi_3\\
\vdots\\
\chi_m\\
\end{bmatrix}.
\]
We already know that each $\chi_i$ is an element of $B_k$. 
We now show that, for each $-m\leq j\leq -1$, the elements $a_{(k,j)}y_{k}^j$ also belong to $B_k$.
It is sufficient to do this by showing that the coefficient matrix 
\[
U:=\begin{bmatrix}
-\xi&-2\xi&\cdots &-m\xi \\
(-\xi)^2&(-2\xi)^{2}&\cdots &(-m\xi)^{2}\\
(-\xi)^3&(-2\xi)^{3}&\cdots &(-m\xi)^{3}\\
\vdots&\vdots&\ddots&\vdots\\
(-\xi)^m&(-2\xi)^{m}&\cdots &(-m\xi)^{m}\\
\end{bmatrix}
\]
is invertible. One can observe that $U$ is essentially a Vandermonde matrix with distinct and nonzero generators (as $\k$ has characteristic $0$), hence it is invertible.  
Therefore, each $a_{(k,j)}y_{k}^j$ is a linear combination of the $\chi_i\in B_k.$ We can therefore conclude that, for each $k\in {\mathfrak s}_{<+\infty}$, the element $a_{(k,j)}y_{k}^j$ is in $B_k$, for all $j<0$. 

Consequently, 
 $x_{-}=\sum_{j=-1}^{-m}a_{(k,j)}y_{k}^j\in B_k,$ and so $x=x_++x_-\in B_k$ for each $k\in  {\mathfrak s}_{<+\infty}$.  
Finally,  
$$x\in \bigcap_{k\in  {\mathfrak s}_{<+\infty} }B_k= \bigcap_{k\in  {\mathfrak s}_{<+\infty} }RE_{{\mathfrak s}_{<+\infty}\setminus\{k\} }^{-1}= R,$$ 
the last equality following from Proposition~\ref{b4}. 
\end{proof} 
\begin{cor}
\label{cor-inner} 
$\ad_x$ and $\theta=D-\ad_x$ are Poisson derivations of $R$.
\end{cor}

\subsection{Homogeneous Poisson derivations}\label{SS:homogeneous}

Our next task is to study how the Poisson derivation $\theta$ appearing in the decomposition~\eqref{E:D-as-ad-plus-central} acts on the generators $x_1, \ldots, x_N$ of $R$. We refer to $\theta$ as a {\em homogeneous Poisson derivation}, a terminology that will be justified in the sequel. 

Recall that the action of $\theta$ on the homogeneous elements $y_i$, with $i\in [1,N]$, is of the form 
\[
\theta(y_k)=\begin{cases}
\omega_k y_k & \text{if $y_k\notin\ZC_P(R)$,}\\
\omega_k & \text{if $y_k\in\ZC_P(R)$,}
\end{cases}
\]
for some $\omega_1, \ldots, \omega_N\in \ZC_P(R)$.

From the definition of a the PNA, $R=R_{k-1}[x_{k}; \sigma_{k}, \delta_{k}]_P\cdots [x_N; \sigma_N, \delta_N]_P$ is equipped with the action of the maximal torus $\mathcal{H}=(\k^*)^n$ by Poisson automorphisms, where $n$ is the rank of $R$ and the character group $X(\mathcal{H})$ is isomorphic to $\chi:=\Z^n$
(see \cite[Proof of Lemma 6.7]{goodearl2023cluster}). For each homogeneous $x\in R$, there exists a unique {\em weight} $\underline\alpha =(\alpha_1, \ldots, \alpha_n)\in \chi$ such that
$$\underline h\cdot x=h_1^{\alpha_1}\ldots h_n^{\alpha_n}x,$$
for all 
$\underline h=(h_1, \ldots, h_n)\in \mathcal{H}$.

We are now ready to describe the action of the homogeneous Poisson derivation $\theta$ on the generators $x_i$ (and on all homogeneous elements). We omit the proof of the next result as it is identical, up to the corresponding adaptations, to that of~\cite[Proposition 5.6]{LLO25}, taking into account the comments that precede it there. For the interested reader, we do mention that the proof uses Proposition~\ref{b5}, whose proof was given here as it is simpler in the Poisson setting than in the case of quantum nilpotent algebras which was considered in~\cite{LLO25}.

\begin{pro}
\label{centralderivation1}
Let $\theta$ be a Poisson $\k$-derivation of $R$. Assume that: 
\begin{enumerate}
\item None of the generators $x_i$, with $i\in [1,N]$, is Poisson-central in $R$. 
\item There are $\omega_1, \ldots, \omega_N\in \ZC_P(R)$ such that $
\theta(y_k)=\begin{cases}
\omega_k y_k & \text{if $y_k\notin\ZC_P(R)$;}\\
\omega_k & \text{if $y_k\in\ZC_P(R)$.}
\end{cases}
$
\end{enumerate}
Then, there exists an abelian group homomorphism $\eta:\chi\rightarrow \ZC_P(R)$
such that, for every homogeneous element $a\in R$, 
\[\theta(a)=\eta(\wt(a))a.\]
\end{pro}

Since $x_k$ and $y_k$ are homogeneous elements of $R$, for each $k\in [1,N]$, we have the following immediate corollary.   
\begin{cor}
\label{theta(y_k)inZ(R)y_k}
Under our running assumptions,
$\theta(x_k)\in \ZC_P(R)x_k$ and $\theta(y_k)\in \ZC_P(R)y_k$, for all $k\in [1, N].$ 
\end{cor}

We are now in position to describe the first Poisson cohomology group of $R$. Recall that $\hh(R)=\der(R)/\innder(R)$.

\begin{theorem}
\label{thm-Der-QNA}
Let $R=\k[x_1]_P[x_2;\sigma_2,\delta_2]_P\cdots
[x_N;\sigma_N,\delta_N]_P$ be a uniparameter PNA of rank $n$. Assume that: 
\begin{enumerate}
\item None of the generators $x_i$, with $i\in [1,N]$, is Poisson-central in $R$;
\item $\ZC_P(R)=\NS_P(R)$.
\end{enumerate}
Then every Poisson derivation $D$ of $R$ can be uniquely written as $D=\ad_x + \theta_{\eta}$, where $x \in R$ and $\theta _{\eta}$ is the homogeneous Poisson derivation of $R$ associated to the abelian group homomorphism $\eta: \chi\rightarrow \ZC_P(R)$ defined by $\theta_\eta (a) = \eta (\wt(a))a$, for any homogeneous element $a \in R$. 
\end{theorem}
\begin{proof}
The existence part follows from Lemma~\ref{L:inner:x-in-R}, Corollary~\ref{cor-inner} and Proposition~\ref{centralderivation1}. The unicity part follows from the unicity of the decomposition of a Poisson derivation of a Poisson torus as the sum of an inner Poisson derivation and a central Poisson derivation, by \cite[Theorem 2.6]{lo} since we can (uniquely) extend any Poisson derivation of $R$ to a Poisson derivation of the Poisson torus $\mathcal{T}$. 
\end{proof}

Under the above hypothesis, we can thus conclude the following.

\begin{cor}
\label{cor-HH1-QNA}
$\hh(R)$ is a free $\ZC_P(R)$-module of rank $\rk(\chi)=\rk(R)=n$. 
\end{cor}

The above results apply in particular to the coordinate rings of $\mathcal{O}^{\mathbf{u}}$ when $\bf{u} =w_0=-1$ (see Example~\ref{E:B-S}).\\

We note that our results combined with the main result of \cite{LLO25} show that, for a QNA $\mathcal{R}$ and its PNA semi-classical limit $R$, we have $\mathrm{HH}^1(\mathcal{R}) \simeq \mathrm{HP}^1(R)$. It would be interesting to compare higher cohomology groups as well. We intend to come back to this question in future work.

\section*{Acknowledgements}

S.\ Launois and S.\ Lopes acknowledge support from the bilateral cooperation PESSOA Program between France and Portugal, reference numbers \texttt{PHC PESSOA 2025 -  54952PC} and \texttt{2025.08159.CBM}.

S.\ Lopes was partially supported by CMUP -- Centro de Matem\'atica da Universidade do Porto, member of LASI, which is financed by national funds through FCT -- Funda\c c\~ao para a Ci\^encia e a Tecnologia, I.P., under the project with reference UID/00144/2025, doi: \url{https://doi.org/10.54499/UID/00144/2025}.

I.\ Oppong acknowledges support from the London Mathematical Society through a Research in Pairs Grant (Ref. 42505) to visit  S.\ Launois at the Université de Caen Normandie, France.

\bibliographystyle{amsalpha}

\end{document}